\documentclass[11pt]{amsart}
\usepackage{amsmath,amsfonts,amssymb,amsthm}
\usepackage{comment}
\usepackage{bbm}
\usepackage{mathtools}
\usepackage{enumerate}
\usepackage[pagebackref,colorlinks,citecolor=blue,linkcolor=blue]{hyperref}
\usepackage[dvipsnames]{xcolor}

\usepackage{geometry}
\usepackage{soul}
\newcommand\rep[2]{{\color{red}\st{#1}#2}}

\theoremstyle{plain}
\newtheorem{theorem}{Theorem}[section]
\newtheorem{corollary}[theorem]{Corollary}
\newtheorem{lemma}[theorem]{Lemma}
\newtheorem{proposition}[theorem]{Proposition}

\theoremstyle{definition}

\newtheorem{example}[theorem]{Example}
\newtheorem{remark}[theorem]{Remark}

\numberwithin{equation}{section}

\newcommand{\R}{{\mathbb R}}
\newcommand{\N}{{\mathbb N}}

\newcommand{\Om}{\Omega}

\providecommand{\vint}[1]{\mathchoice
          {\mathop{\vrule width 5pt height 3 pt depth -2.5pt
                  \kern -9pt \kern 1pt\intop}\nolimits_{\kern -5pt{#1}}}
          {\mathop{\vrule width 5pt height 3 pt depth -2.6pt
                  \kern -6pt \intop}\nolimits_{\kern -3pt{#1}}}
          {\mathop{\vrule width 5pt height 3 pt depth -2.6pt
                  \kern -6pt \intop}\nolimits_{\kern -3pt{#1}}}
          {\mathop{\vrule width 5pt height 3 pt depth -2.6pt
                  \kern -6pt \intop}\nolimits_{\kern -3pt{#1}}}}

\newcommand{\eps}{\varepsilon}
\newcommand\loc{{\rm loc}}

\newcommand{\BV}{\mathrm{BV}}

\newcommand{\mres}{\!\mathbin{\vrule height 1.6ex depth 0pt width
0.13ex\vrule height 0.13ex depth 0pt width 1.1ex}\!}

\newcommand{\yn}{z}
\newcommand{\ys}{y}

\DeclareMathOperator{\rcapa}{cap}

\DeclareMathOperator{\diam}{diam}
\DeclareMathOperator{\dimens}{dim}

\def\XXint#1#2#3{{\setbox0=\hbox{$#1{#2#3}{\int}$}
\vcenter{\hbox{$#2#3$}}\kern-.5\wd0}}

\allowdisplaybreaks[3]

\begin{document}
\title{A note on the weak* and pointwise convergence
	of $\BV$ functions}
\author{Lisa Beck}
\address{{\color{blue}L. B.:} Institut f\"{u}r Mathematik, Universit\"{a}t Augsburg, Universit\"{a}tsstr. 14, 86159~Augsburg, Germany}
\email{lisa.beck@math.uni-augsburg.de}

\author{Panu Lahti}
\address{{\color{blue}P. L.:} Institut f\"{u}r Mathematik, Universit\"{a}t Augsburg, Universit\"{a}tsstr. 14, 86159~Augsburg, Germany \vspace{-8pt}}
\address{Academy of Mathematics and Systems Science, Chinese Academy of Sciences,
Beijing 100190, PR China}
\email{panulahti@amss.ac.cn}

\maketitle
\vspace{-0.5cm}
\begin{abstract} 
We study pointwise convergence properties of weakly* convergent sequences $\{u_i\}_{i \in \N}$ in $\BV(\R^n)$. We show that, after passage to a suitable subsequence (not relabeled), we have pointwise convergence $u_i^*(x)\to u^*(x)$ of the precise representatives for all $x\in \R^n \setminus E$, where the exceptional set $E \subset \R^n$ has on the one hand Hausdorff dimension at most~$n-1$, and is on the other hand also negligible with respect to the Cantor part of~$|D u|$. Furthermore, we discuss the optimality of these results.
\end{abstract}

\section{Introduction}

Let $N,n\in\N$ and consider a sequence $\{u_i\}_{i \in \N}$ of functions in $\BV(\R^n;\R^N)$. We are interested in studying pointwise convergence properties under different assumptions of convergence on the sequence. In this regard, let us recall that for every function $u \in \BV(\R^n)$ its Lebesgue representative~$\widetilde{u}$ is well-defined outside of the approximate discontinuity set~$S_u$ (which is of Hausdorff dimension at most~$n-1$), while for a finer analysis one works with the precise representative $u^*$ which provides a well-defined extension of~$\widetilde{u}$ to the jump set $J_u \subset S_u$ (and the remaining set $S_u \setminus J_u$ is negligible with respect to the $(n-1)$-dimensional Hausdorff measure~$\mathcal H^{n-1}$), see Section~\ref{sec:preliminaries} for the precise statement. If we assume the strong convergence $u_i\to u$ in $L^1(\R^n;\R^N)$, then  it is well-known that for a (not relabeled) subsequence we have $\widetilde{u}_i(x)\to \widetilde{u}(x) $ for $\mathcal L^n$-almost every $x\in\R^n$. If we even have strong convergence $u_i\to u$ in $\BV(\R^n;\R^N)$, then for a (not relabeled) subsequence we have $u_i^*(x)\to u^*(x)$ for $\mathcal H^{n-1}$-almost every $x\in \R^n$ (which follows e.g. from \cite[Remark 4.1, Lemma 4.2]{LaSh}).

Here we investigate what can be said about pointwise convergence if the sequence $\{u_i\}_{i \in \N}$ is known to converge to~$u$ in a stronger topology than $L^1(\R^n;\R^N)$, but a weaker one than $\BV(\R^n;\R^N)$. Mostly, we are interested in the case of weak* convergence $u_i\to u$ in $\BV(\R^n;\R^N)$. For this purpose, we proceed in two different directions. First, we follow an approach via capacity estimates and prove for a subsequence that pointwise convergence holds outside of an exceptional set $E \subset \R^n$ of Hausdorff dimension at most~$n-1$.

\begin{theorem}\label{thm:Hausdorff main BV theorem}
\label{thm_bounded_BV_cap}
Let $u \in \BV(\R^n)$. Let $\{u_i\}_{i \in \N}$ be a sequence in $\BV(\R^n)$ 
for which $\{|D u_i|(\R^n)\}_{i \in \N}$ is bounded, and
suppose that $\widetilde{u}_i(x) \to \widetilde{u}(x)$ for $\mathcal L^n$-almost
every $x\in \R^n$. Then there exists a set $E \subset \R^n$ such that 
$\dim_{\mathcal H}(E) \leq n-1$ and such that for a subsequence 
\textup{(}not relabeled\textup{)} we have
\[
\widetilde{u}_i(x) \to \widetilde{u}(x) \quad \text{for every } x \in \R^n \setminus E.
\]
\end{theorem}

This result is analogous to the case of weakly convergent sequences in a Sobolev space $W^{1,p}(\R^n)$ with $p \geq 1$, where pointwise convergence of a subsequence holds outside of an exceptional set of Hausdorff dimension at most~$n-p$. The proof of Theorem~\ref{thm_bounded_BV_cap} is carried out in Section~\ref{sec_pointwise_convergence_capacity_method} (along with the proof of the corresponding result for weakly convergent sequences in Sobolev spaces as mentioned above, which is included for comparison) and it is essentially based on capacity estimates for fractional Sobolev spaces and interpolation arguments.

Second, we address pointwise convergence with respect to a diffuse measure $|D^d w|$, which is defined as the sum of the absolutely continuous and the Cantor part of the variation measure~$|Dw|$ for a function $w\in\BV(\R^n)$. We prove for a subsequence (that might depend on the choice of~$w$) that pointwise convergence holds outside of an exceptional set $E \subset \R^n$ of vanishing $|D^d w|$-measure.

\begin{theorem}\label{thm:main BV theorem}
Let $u\in\BV(\R^n)$. Let $\{u_i\}_{i \in \N}$ be a sequence in $\BV(\R^n)$ 
for which $\{|D u_i|(\R^n)\}_{i \in \N}$ is bounded, and
suppose that $u_i^*(x)\to u^*(x)$ for $\mathcal L^n$-almost
every $x\in \R^n$. Let $w\in\BV(\R^n)$.
Then for a subsequence \textup{(}not relabeled\textup{)} we
have 
\[
u_i^*(x)\to u^*(x) \quad  \text{for $|D^d w|$-almost every } x\in \R^n.
\]
\end{theorem}

We notice that the assumption of pointwise convergence $\mathcal L^n$-almost everywhere is of course satisfied (for a subsequence) in the particular situation where we have weak* convergence in $\BV(\R^n)$, see also Remark~\ref{remark_assumptions_weak_star_convergence}. Furthermore,
the pointwise convergence with respect to the absolutely continuous part follows of course already from the $\mathcal L^n$-almost everywhere convergence, but the pointwise convergence with respect to the Cantor part~$|D^c w|$ may contain additional information when compared to Theorem~\ref{thm_bounded_BV_cap} since~$|D^c w|$ can be supported on an $(n-1)$-dimensional set. The proof of Theorem~\ref{thm:main BV theorem} is executed in Section~\ref{sec_pointwise_convergence_slicing} and relies heavily on the theory of one-dimensional sections of $\BV$ functions given e.g. in \cite[Section 3.11]{AFP}.

\begin{remark}
The results of Theorem~\ref{thm_bounded_BV_cap} and Theorem~\ref{thm:main BV theorem} hold also in the vector-valued case where the sequence $\{u_i\}_{i \in \N}$ and the limit function~$u$ are
taken in the space $\BV(\R^n;\R^N)$ with $N \in \N$. This is seen easily by considering the component functions.
\end{remark}

In Example \ref{ex:liftings} we discuss a possible application of our results
in the context of \emph{\textup{(}approximable\textup{)} liftings}.
Let us finally observe that the results of Theorem~\ref{thm_bounded_BV_cap} and Theorem~\ref{thm:main BV theorem} are sharp in the sense that the exceptional set~$E$, where pointwise convergence fails, in general does not satisfy $\mathcal H^{n-1}(E) =0$ or $|D w|(E)=0$. In particular, we give examples in Section~\ref{sec_remarks_examples} which demonstrate that the jump set $D^j w$ needs to be excluded in the statement of the pointwise convergence and that also the passage to a subsequence is in general necessary. We further discuss
some aspects of the possible size of the exceptional set~$E$,
and analyze two particular situations in Section~\ref{sec:decreasing}.

\section{Notation and preliminaries}
\label{sec:preliminaries}

\subsection{General notation}

As already mentioned in the introduction, we consider $N,n\in\N$ and we will always work 
in the space~$\R^n$. The matrix space $\R^{N\times n}$ will always be
equipped with the Euclidean norm $|A|\coloneqq (\sum_{i=1}^N \sum_{j=1}^n A_j^i)^{1/2}$,
where~$i$ and~$j$ are the row and column indices, respectively. 
For $a\in\R^N$ and $b\in\R^n$, we define the tensor product
$a\otimes b \coloneqq a b^T\in \R^{N\times n}$, where $a,b$ are considered
as column vectors. We write~$B(x,r)$ for the open ball in~$\R^{n}$ with center~$x$
and radius~$r$, that is, $\{\yn \in \R^n \colon |\yn-x|<r\}$, and we write~$\mathbb S^{n-1}$ 
for the unit sphere in~$\R^n$, that is, $\{\yn \in \R^n \colon |\yn| =1\}$. For a set 
$S \subset \R^n$ we use the notation $S^\mathrm{o}$ to indicate the topological interior. 

We denote the $n$-dimensional Lebesgue measure by $\mathcal L^n$ and use the abbreviation 
$\omega_n = \mathcal L^n(B(0,1))$, and we denote the $s$-di\-men\-sio\-nal Hausdorff measure 
by~$\mathcal H^s$. Given any measure~$\nu$ on~$\R^n$, the restriction of
$\nu$ to a set $S \subset \R^n$ is denoted by $\nu\mres S$, that is,
$\nu\mres S(B) \coloneqq \nu(S\cap B)$ for all sets $B \subset \R^n$. 
The Borel $\sigma$-algebra on a set $S\subset \R^n$ is denoted by $\mathcal B(S)$.

For a function $u$ we write $u_+ \coloneqq \max\{u,0\}$ for its positive part, and 
if it is integrable on some measurable set $S \subset \R^n$ of positive and finite Lebesgue measure, we write $\vint{S} u \,d\mathcal L^n \coloneqq (\mathcal L^n(S))^{-1} \int_S u(x) \,d\mathcal L^n(x)$
for its mean value on~$S$. We further denote by~$\mathbbm{1}_S$ the characteristic function of the set~$S$.

\subsection{Fractional Sobolev spaces and capacity}

Let $\Omega \subset \R^n$ be an open set,  $s \in (0,1)$ and $p \in [1,\infty)$. A function $u \in L^p(\Omega)$ is said to belong to the \emph{fractional Sobolev space} $W^{s,p}(\Omega)$ if $(x,\yn) \mapsto |u(x)-u(\yn)| |x - \yn|^{-n/p-s} \in L^p(\Omega \times \Omega)$. When $W^{s,p}(\Omega)$ is endowed with the norm
\begin{equation*}
 \| u \|_{W^{s,p}(\Omega)} \coloneqq  \Big( \| u \|_{L^{p}(\Omega)}^p + [u]_{W^{s,p}(\Omega)}^p  \Big)^{\frac{1}{p}},
\end{equation*}
where
\begin{equation*}
 [u]_{W^{s,p}(\Omega)}^p \coloneqq  \int_\Omega \int_{\Omega} \frac{|u(x) - u(\yn)|^p}{| x -\yn|^{n+sp}} \,d\mathcal L^n(x) \,d\mathcal L^n(\yn),
\end{equation*}
it is a Banach space. The number $s$ can be interpreted as fractional differentiability, and in some sense the fractional Sobolev spaces $W^{s,p}$ are interpolation spaces between the classical Sobolev space $W^{1,p}$ and the Lebesgue space $L^p$. As a matter of fact, many properties known from classical Sobolev spaces extend to fractional Sobolev spaces. We here comment only on a few on them, which are relevant for our paper, and refer for instance to~\cite{ADAMS75,DINPALVAL12} for a detailed discussion. In particular, we have the inclusion $W^{s',p}(\Omega) \subset W^{s,p}(\Omega)$ for every $s' \in (s,1)$, which continues to hold also for the classical Sobolev space with $s'=1$ if $\Omega=\R^n$ of if $\Omega$ is a bounded Lipschitz domain, see e.g. \cite[Proposition 2.1 \& Proposition 2.2]{DINPALVAL12}. Moreover, as in the case of classical Sobolev spaces with integer differentiability and still in the case of bounded Lipschitz domains, there exists a linear, bounded extension operator from $W^{s,p}(\Omega)$ to $W^{s,p}(\R^n)$, see \cite[Theorem 5.4]{DINPALVAL12}. Furthermore, the space $C^\infty_c(\R^n)$ of smooth functions with compact support is dense in $W^{s,p}(\R^n)$, see \cite[Theorem 7.38]{ADAMS75}. 

Associated to the classical ($s=1$) and fractional ($s \in (0,1)$) Sobolev norm, we also introduce the \emph{$(s,p)$-Sobolev capacity} of a set $E \subset \R^n$ as 
\begin{equation*}
 \rcapa_{s,p}(E) \coloneqq  \inf \big\{ \| u \|_{W^{s,p}(\R^n)} \colon u \in W^{s,p}(\R^n) \text{ with } E \subset \{ u \geq 1\}^\mathrm{o} \big\}
\end{equation*}
(which as usual is interpreted as $\infty$ if there doesn't exist any function $u \in W^{s,p}(\R^n)$ with $E \subset \{ u \geq 1\}^\mathrm{o}$). Note that it is easy to verify from its definition that $\rcapa_{s,p}$ is an outer measure on $\R^n$, i.e., it assigns zero measure to the empty set, it is monotone, and it is countably subadditive. Moreover, there holds ${\mathcal L}^n(E) \leq \rcapa_{s,p}(E)$ for all sets $E \subset \R^n$, i.e., the Sobolev capacity measure $\rcapa_{s,p}$ is a finer measure compared to the Lebesgue measure ${\mathcal L}^n$. It is also not difficult to verify that $\rcapa_{s,p}(\{x\})>0$ holds for every single point $x \in \R^n$ whenever $sp>n$, hence, there are no nontrivial sets of vanishing $\rcapa_{s,p}$-measure, which implies that the $(s,p)$-capacity is only useful if $sp \leq n$. Let us further notice that for sets of vanishing $\rcapa_{s,p}$-measure we also have an immediate upper bound on their Hausdorff dimension:

\begin{proposition}
\label{prop_cap_Hausdorff_dim}
Let $s \in (0,1)$ and $p \in (1,\infty)$ with $sp \leq n$. If a set $E\subset \R^n$ satisfies $\rcapa_{s,p}(E)=0$, then we have $\dim_{\mathcal H}(E) \leq n -sp$.  
\end{proposition}

\begin{proof}
We essentially follow the first part of the proof of \cite[Section 4.7, Theorem 4]{EvGa}). Since by assumption $\rcapa_{s,p}(E)=0$ holds, we find a sequence of functions $\{u_i\}_{i \in \N}$ in $W^{s,p}(\R^n)$ such that 
\[
\| u_i \|_{W^{s,p}(\R^n)} \leq 2^{-i} \quad \text{and} \quad E \subset \{ u_i \geq 1\}^\mathrm{o}
\]
are satisfied for each $i \in \N$. The first condition ensures that $v \coloneqq  \sum_{i \in \N} u_i$ defines a function in $W^{s,p}(\R^n)$, while the second condition implies
\[
 E \subset \bigg\{ x \in \R^n \colon \limsup_{r \to \infty} \vint{B(x,r)} v(\yn) \,d\mathcal L^n(\yn) = \infty \bigg\}.  
\] 
Since in view of~\cite[Proposition 1.76]{BEC16} the set on the right-hand side has Hausdorff dimension at most $n-sp$, the claim $\dim_{\mathcal H}(E) \leq n -sp$ is established.
\end{proof}

\subsection{Radon measures}
Let $\Om\subset\R^n$ be an open set and $\ell \in\N$. We denote by $C_c(\Om;\R^{\ell})$ 
the space of continuous $\R^{\ell}$-valued functions with compact
support in $\Om$ and by $C_0(\Om;\R^{\ell})$ its completion with
respect to the $\Vert \cdot\Vert_{\infty}$-norm.
We further denote by $\mathcal M(\Om;\R^{\ell})$ the Banach space of vector-valued
Radon measures, equipped with the
\emph{total variation norm} $|\mu|(\Om)<\infty$, which is defined
relative to the Euclidean norm on $\R^{\ell}$. By the Riesz
representation theorem, $\mathcal M(\Om;\R^{\ell})$ is the dual space of $C_0(\Om;\R^{\ell})$, with the duality pairing $\langle\phi,\mu\rangle\coloneqq \int_{\Om} \phi\cdot d\mu\coloneqq \sum_{j=1}^{\ell} 
\int_{\Om} \phi_j \,d\mu_j$.
Thus, weak* convergence $\mu_i\overset{*}{\rightharpoondown}\mu$
in $\mathcal M(\Om;\R^{\ell})$ means $\langle\phi,\mu_i\rangle\to \langle\phi,\mu\rangle$ for all $\phi\in C_0(\Om;\R^{\ell})$.
We further denote the set of positive measures by $\mathcal M^+(\Om)$.

For a vector-valued Radon measure $\gamma\in\mathcal M(\Om;\R^{\ell})$
and a positive Radon measure $\mu\in\mathcal M^+(\Om)$, we
can write the \emph{Lebesgue--Radon--Nikod\'ym decomposition}
\[
\gamma=\gamma^a+\gamma^s=\frac{d\gamma}{d\mu}\,d\mu+\gamma^s
\]
of $\gamma$ with respect to $\mu$,
where $\frac{d\gamma}{d\mu}\in L^1(\Om,\mu;\R^{\ell})$.

For open sets $E\subset \R^{n-m}$, $F\subset \R^m$ and $m \in \{1,\ldots,n-1\}$, a \emph{parametrized measure} $(\nu_{\ys})_{\ys \in E}$ is a mapping from $E$ to the set $\mathcal M(F;\R^{\ell})$ of vector-valued Radon measures on $F$. It is said to be \emph{weakly* $\mu$-measurable}, for $\mu\in\mathcal M^+(E)$, if $\ys\mapsto \nu_{\ys}(B)$ is $\mu$-measurable for all Borel sets $B\in\mathcal{B}(F)$ (it suffices to check this for open subsets). Equivalently, $(\nu_{\ys})_{\ys\in E}$ is weakly* $\mu$-measurable if the function $\ys \mapsto \int_F f(\ys,t)\,d\nu_x(t)$ is $\mu$-measurable for every bounded
$\mathcal B_\mu(E) \times \mathcal B(F)$-measurable function $f\colon E\times F\to \R$ (see \cite[Proposition 2.26]{AFP}), where $\mathcal B_\mu(E)$ denotes the $\mu$-completion of $\mathcal B(E)$. Suppose that we additionally have
\[
\int_{E}|\nu_{\ys}|(F)\,d\mu(\ys)<\infty.
\]
In that case we denote by $\mu \otimes \nu_\ys$ the generalized product measure defined by
\[
\mu\otimes\nu_\ys(A)\coloneqq \int_E \left(\int_F \mathbbm{1}_A(\ys,t)\,d\nu_\ys(t)\right)\,d\mu(\ys)
\]
for any $A\in \mathcal{B}(E\times F)$. Then the integration formula
\[
\int_{E\times F}f(\ys,t)\,d(\mu\otimes \nu_\ys)(\ys,t)=\int_E\left(\int_F f(\ys,t)\,d\nu_\ys(t)\right)\,d\mu(\ys)
\]
holds for every bounded Borel function $f\colon E\times F\to\R$ and,
in the case $\ell=1$ and $\nu_\ys\ge 0$, if $f$ is a positive Borel function.

\subsection{Functions of bounded variation}

Let $\Omega\subset \R^n$ be an open set. A function
$u\in L^1(\Omega;\R^N)$ is said to belong to the space $\BV(\Omega;\R^N)$ of 
\emph{functions of bounded variation} if its distributional derivative
is a finite $\R^{N\times n}$-valued Radon measure. This means that
there exists a (unique) measure $Du\in \mathcal M(\R^n;\R^{N\times n})$
such that for all functions $\psi\in C_c^1(\Omega)$, the integration-by-parts formula
\[
\int_{\Omega}\frac{\partial\psi}{\partial x_k}u^j\,d\mathcal L^n=-\int_{\Omega}\psi\,dDu_k^j, \quad j=1\,\ldots N,\ \ k=1,\ldots,n
\]
holds. The space $\BV(\Omega;\R^N)$ is a Banach space endowed with the norm
\[
\Vert u\Vert_{\BV(\Omega;\R^N)}\coloneqq \Vert u\Vert_{L^1(\Omega;\R^N)}+|Du|(\Omega).
\]

The theory of $\BV$ functions presented here can be found in~\cite{AFP} 
(see also~\cite{EvGa,Fed,Zie}), and we will give specific references only for a few key 
results relevant for our paper. 

Let us first note that for a scalar-valued function $u \in \BV(\Omega)$ the total variation of $Du$
can be obtained by integration over the super-level sets $S_t\coloneqq \{x\in\Omega \colon u(x)>t\}$, $t\in\R$, 
via the coarea formula (see \cite[Theorem 3.40]{AFP}) as
\begin{equation}
\label{eq_coarea}
|Du|(\Omega)=\int_{-\infty}^{\infty}|D\mathbbm{1}_{S_t}|(\Omega)\,dt.
\end{equation}

We now recall different notions of convergence for sequences in $\BV(\Omega;\R^N)$.
We say that a sequence of functions $\{u_i\}_{i \in \N}$ in 
$\BV(\Omega;\R^N)$ \emph{converges weakly*} to $u\in\BV(\Omega;\R^N)$, denoted by
$u_i\overset{*}{\rightharpoondown} u$ in $\BV(\Omega;\R^N)$, if $u_i\to u$ strongly in 
$L^1(\Omega;\R^N)$ and $Du_j\overset{*}{\rightharpoondown}Du$ in $\mathcal M(\Omega;\R^{N\times n})$.
Note that every weakly* convergent sequence $\{u_i\}_{i \in \N}$ in $\BV(\Omega;\R^N)$ is norm-bounded
by the Banach--Steinhaus theorem. Conversely, every norm-bounded sequence $\{u_i\}_{i \in \N}$ in 
$\BV(\Omega;\R^N)$ with strong convergence $u_i\to u$ in $L^1(\Omega;\R^N)$ satisfies 
$u_i\overset{*}{\rightharpoondown} u$ in $\BV(\Omega;\R^N)$, see \cite[Proposition 3.13]{AFP}. Moreover, 
every norm-bounded sequence in $\BV(\Omega;\R^N)$ has a weakly* convergent subsequence if~$\Omega$ 
is sufficiently regular, e.g.~a bounded Lipschitz domain. Moreover, we say that a sequence of functions $\{u_i\}_{i \in \N}$ in $\BV(\Omega;\R^N)$ \emph{converges strictly} to $u\in\BV(\Omega;\R^N)$ if $u_i\to u$ strongly in $L^1(\Omega;\R^N)$ and $|Du_i|(\Omega)\to |Du|(\Omega)$.

We have the following simple fact concerning weak* convergence in the $\BV$ class.

\begin{proposition}\label{prop:pointwise to L1 convergence}
Let $u\in\BV(\R^n)$. Let $\{u_i\}_{i \in \N}$ be a sequence in $\BV(\R^n)$ for which 
$\{|D u_i|(\R^n)\}_{i \in \N}$ is bounded, and suppose that $u_i(x)\to u(x)$ 
for $\mathcal L^n$-almost every $x\in \R^n$. Then we have 
$u_i\to u$ in $L^q_{\loc}(\R^n)$ for every $q \in [1,1^*)$, where $1^*$ is the Sobolev conjugate of~$1$ defined as $\tfrac{n}{n-1}$ if $n>1$ and $\infty$ if $n=1$.
\end{proposition}

\begin{proof}
Fix an arbitrary $R>0$. For sufficiently large $i \in \N$, we have $|u_i-u|\le 1$ in some set 
$A\subset B(0,R)$ with
\[
\frac{\mathcal L^n(A)}{\mathcal L^n(B(0,R))}\ge \frac{1}{2}.
\]
Therefore, by a Poincar\'e inequality (see e.g. \cite[Lemma 2.2]{KKLS}) there holds
\begin{align*}
\int_{B(0,R)}|u_i-u|\,d\mathcal L^n
	&\le \int_{B(0,R)}(|u_i-u|-1)_+\,d\mathcal L^n+\mathcal L^n(B(0,R))\\
	&\le C(n)R|D(u_i-u)|(B(0,R))+\mathcal L^n(B(0,R))\\
	&\le C(n)R(|Du_i|(\R^n)+|Du|(\R^n))+\mathcal L^n(B(0,R)).
\end{align*}
Thus, $\{u_i-u\}_{i \in \N}$ is a bounded sequence in $\BV(B(0,R))$. Let $q \in [1,1^*)$ be arbitrary. By the Sobolev embedding and the Rellich--Kondrachov compactness theorem, we get for the sequence $\{u_i\}_{i \in \N}$ boundedness in $L^{1^*}(B(0,R))$ and, for a subsequence, strong convergence in $L^{q}(B(0,R))$, necessarily to~$0$. Since from every subsequence of $\{u_i\}_{i \in \N}$ we can choose a further subsequence converging to~$u$ in $L^q(B(0,R))$, this is in fact true also for the original sequence. Since $R>0$ was arbitrary, we have shown the convergence  $u_i\to u$ in $L^q_{\loc}(\R^n)$ as claimed.
\end{proof}


Let $u\in L^1_{\rm loc}(\R^n;\R^N)$. We say that $u$ has a Lebesgue point at $x\in\R^n$ if
\[
\lim_{r\to 0} \; \vint{B(x,r)}|u(\yn)-\widetilde{u}(x)|\,d\mathcal L^n(\yn)=0
\]
for some (unique) $\widetilde{u}(x)\in\R^N$. We denote by $S_u$ the set where this condition fails and call it the \emph{approximate discontinuity set}. We note that $S_u$ is a Borel set of vanishing $\mathcal L^{n}$-measure and that $\widetilde{u} \colon \R^n \setminus S_u \to \R^N$ is Borel-measurable, see \cite[Proposition 3.64]{AFP}.

Given $\nu\in \mathbb S^{n-1}$, we denote the upper and lower half-ball with respect to $\nu$ as
\begin{align*}
B_{\nu}^+(x,r)\coloneqq \{\yn \in B(x,r) \colon \langle \yn-x,\nu\rangle>0\},\\
B_{\nu}^-(x,r)\coloneqq \{\yn \in B(x,r) \colon \langle \yn-x,\nu\rangle<0\}.
\end{align*}
We say that $x\in \R^n$ is an approximate jump point of $u$ if there exist
$\nu\in \mathbb S^{n-1}$ and $u^+(x), u^-(x)\in\R^N$ with
$u^+(x) \neq u^-(x)$ (called the one-sided approximate limits) such that
\[
\lim_{r\to 0} \; \vint{B_{\nu}^+(x,r)}|u(\yn)-u^+(x)|\,d\mathcal L^{n}(\yn)=0
\]
and
\[
\lim_{r\to 0} \; \vint{B_{\nu}^-(x,r)}|u(\yn)-u^-(x)|\,d\mathcal L^{n}(\yn)=0.
\]
We denote by~$J_u$ the set of approximate jump points and call it the \emph{approximate jump set}. If $u \in \BV(\Omega;\R^N)$, then $S_u$ is countably $\mathcal H^{n-1}$-rectifiable and $\mathcal H^{n-1}(S_u\setminus J_u)=0$, see \cite[Theorem 3.78]{AFP}. 

For a finer analysis of $\BV$ functions, we can now define the \emph{precise representative} of $u \in \BV(\Omega;\R^N)$ as
\[
u^*(x)\coloneqq 
\begin{cases}
\widetilde{u}(x) & \quad \textrm{if }x\in\R^n\setminus S_u,\\
(u^+(x)+u^-(x))/2 & \quad \textrm{if }x\in J_u,
\end{cases}
\]
which is uniquely determined $\mathcal H^{n-1}$-almost everywhere. We then write the Lebesgue--Radon--Nikod\'ym decomposition of the variation measure of $u$ into the absolutely continuous and singular parts as $Du=D^a u+D^s u$. Furthermore, we define the jump and Cantor parts of $Du$ by
\[
D^j u\coloneqq D^s u\mres J_u,\qquad  D^c u\coloneqq  D^s u\mres (\R^n\setminus S_u).
\]
Since $Du$ vanishes on~$\mathcal H^{n-1}$-negligible sets (see \cite[Lemma 3.76]{AFP}), we obtain with $\mathcal H^{n-1}(S_u\setminus J_u)=0$ the decomposition
\[
Du=D^a u+ D^c u+ D^j u.
\]
Moreover, we call the sum $D^a u+D^c u$ the \emph{diffuse} part of the variation measure and denote it by $D^d u$.

\subsection{One-dimensional sections of $\BV$ functions}\label{subsec:one dimensional sections}

The following notation and results on one-dimensional sections of
$\BV$ functions, as given in \cite[Section 3.11]{AFP}, will be
crucial for us.

In the one-dimensional case $n=1$, we have $J_u=S_u$, $J_u$ is at most countable,
and $Du (\{x\})=0$ for every $x\in\R\setminus J_u$. Moreover, we have
for every $x,\tilde{x}\in \R\setminus J_u$ with $x<\tilde{x}$ that
\begin{equation}\label{eq:fundamental theorem of calculus for BV}
u^*(\tilde{x})-u^*(x)=Du((x,\tilde{x})),
\end{equation}
and for every $x,\tilde{x}\in \R$ with $x<\tilde{x}$ that
\begin{equation}\label{eq:fundamental theorem of calculus for BV 2}
|u^*(\tilde{x})-u^*(x)|\le |Du|([x,\tilde{x}]).
\end{equation}

In $\R^n$, we denote the standard basis vectors by $e_k$, $k=1,\ldots,n$.
For any fixed $k\in\{1,\ldots,n\}$, $\ys\in\R^{n-1}$
and $t \in \R$, we introduce the notation 
\[
\pi_k(\ys,t) \coloneqq (\ys_1,\ldots,
\ys_{k-1},t,\ys_k,\ldots,\ys_{n-1}) \in \R^n.
\]
For a set $A \subset \R^n$ we then denote the slices of~$A$ at $\pi_k(\ys,0)$ in $e_k$-direction by
\[
A^{e_k}_{\ys}\coloneqq \{t\in\R \colon \pi_k(\ys,t) \in A\}.
\]
For $u\in\BV(\R^n;\R^N)$, we denote $u_\ys^k(t)\coloneqq u(\pi_k(\ys,t))$ and record  that $u_\ys^k\in \BV(\R;\R^N)$ 
is satisfied for $\mathcal L^{n-1}$-almost every $\ys\in\R^{n-1}$  (see \cite[Theorem 3.103]{AFP}). 
Denoting $D_k u\coloneqq \langle Du,e_k\rangle$ and $D^d_k u\coloneqq \langle D^d u,e_k\rangle$ 
(the inner product taken row-wise), we further have
\[
D_k u=\mathcal L^{n-1}\otimes D u^k_\ys \quad 
\text{and} \quad D^d_k u=\mathcal L^{n-1}\otimes D^d u^k_\ys
\]
(see \cite[Theorem 3.107 \& Theorem 3.108]{AFP}.
It follows that
\begin{equation}\label{eq:slice representation for total variation}
|D_k u|=\mathcal L^{n-1}\otimes |D u^k_\ys|  \quad 
\text{and} \quad |D_k^d u|=\mathcal L^{n-1}\otimes |D^d u^k_\ys|
\end{equation}
(see \cite[Corollary 2.29]{AFP}). Moreover, for $\mathcal L^{n-1}$-almost every 
$\ys\in\R^{n-1}$ we have
\begin{equation}\label{eq:sections and jump sets}
J_{u_\ys^k}=(J_u)^k_\ys\quad\textrm{and}\quad
(u^*)_\ys^k(t)=(u^k_\ys)^*(t)\ \ \textrm{for every }t\in \R\setminus J_{u^k_\ys},
\end{equation}
(see \cite[Theorem 3.108]{AFP}).

\section{Pointwise convergence w.r.t.~Hausdorff measures}
\label{sec_pointwise_convergence_capacity_method}

In this section, we first consider strongly convergent sequences in fractional Sobolev spaces $W^{s,p}(\R^n)$, with $s \in (0,1)$ and $p \in (1,\infty)$. For these we establish pointwise convergence outside of a set of vanishing $(s,p)$-capacity, by a straightforward adaptation of the proof for classical Sobolev spaces. We then move on to bounded sequences in the classical Sobolev spaces $W^{1,p}(\R^n)$ and in the space $\BV(\R^n)$ of functions of bounded variation and deduce, via compactness and interpolation results, pointwise convergence up to sets of Hausdorff dimension $n-p$ and $n-1$, respectively, which yields in particular the statement of Theorem~\ref{thm_bounded_BV_cap}.

Let us start by recalling that the exceptional set of non-Lebesgue points of a function in $W^{s,p}(\R^n)$ is of vanishing $(s,p)$-capacity (see e.g. \cite[Theorem 6.2]{LUIVAE17}, combined with the argument from \cite[Section 1.7, Corollary 1]{EvGa}), which is the analogous property as known for classical Sobolev functions (see e.g.~\cite{FEDZIE72} or \cite[Section 4.8, Theorem 1]{EvGa}).

\begin{theorem}
\label{thm_Leb_frac_Sobolev}
Let $s \in (0,1)$, $p \in (1,\infty)$ and $u \in W^{s,p}(\R^n)$. Then there exists a set $E \subset \R^n$ such that $\rcapa_{s,p}(E)=0$ and such that for each $x \in \R^n \setminus E$, for some $\widetilde{u}(x) \in \R$, there holds
\[
\lim_{r\to 0} \vint{B(x,r)}|u(\yn)-\widetilde{u}(x)|\,d\mathcal L^n(\yn)=0 .
\]
\end{theorem}

Next, following the strategy of proof in \cite[Lemma 2.19]{MALZIE97}, we obtain for sufficiently fast convergent sequences in $W^{s,p}(\R^n)$ that also the exceptional set where pointwise convergence fails is of vanishing $(s,p)$-capacity:

\begin{lemma}
Let $s \in (0,1)$, $p \in (1,\infty)$ and $u \in W^{s,p}(\R^n)$. Let $\{u_i\}_{i \in \N}$ be a sequence of functions in $W^{s,p}(\R^n)$ and suppose that $u_i \to u$ strongly in $W^{s,p}(\R^n)$ with
\begin{equation}
\label{eqn_fast_convergence}
 \sum_{i \in \N} 2^{ip} \| u - u_i \|_{W^{s,p}(\R^n)}^p < \infty.
\end{equation}
Then there exists a set $E \subset \R^n$ such that $\rcapa_{s,p}(E)=0$ and such that for each $x \in \R^n \setminus E$ the pointwise limit $\lim_{i \to \infty} \widetilde{u}_i(x)$ of the Lebesgue representatives exists and coincides with $\widetilde{u}(x)$.
\end{lemma}

\begin{proof}
In view of Theorem~\ref{thm_Leb_frac_Sobolev}, we find a set $E_0 \subset \R^n$ with $\rcapa_{s,p}(E_0)=0$ such that there hold
\[
 \widetilde{u}_i(x) = \lim_{r \to 0} \; \vint{B(x,r)} u_i(\yn) \,d\mathcal L^n(\yn) \quad \text{and} \quad \widetilde{u}(x) = \lim_{r \to 0} \; \vint{B(x,r)} u(\yn) \,d\mathcal L^n(\yn) 
\]
for all $x \in \R^n \setminus E_0$ and all $i \in \N$.
For these points $x \in \R^n \setminus E_0$ we then observe
\[
 |\widetilde{u}(x) - \widetilde{u}_i(x)| \leq \sup_{r > 0} \; \vint{B(x,r)} |u(\yn) - u_i(\yn)| \,d\mathcal L^n(\yn) = M(u - u_i)(x),   
\]
where $M$ denotes the Hardy--Littlewood maximal operator. Defining sets 
\[
 A_i \coloneqq  \big\{ x \in \R^n \colon M(u-u_i)(x) > 2^{-i} \big\}
\]
we then deduce from \cite[Lemma 6.4]{LUIVAE17} (based essentially on the facts that $A_i$ is an open set and that the Hardy--Littlewood maximal operator is bounded from $W^{s,p}(\R^n)$ into itself) the estimate
\[
 \rcapa_{s,p}(A_i) \leq C(n,p) 2^{ip} \| u - u_i \|_{W^{s,p}(\R^n)}^p.
\]
Setting 
\[
 E \coloneqq  E_0 \cup \bigcap_{j \in \N} \bigcup_{i \geq j} A_i,
\]
we can then verify the assertions of the lemma. First, by the choice of $E_0$ and by assumption~\eqref{eqn_fast_convergence} we have
\begin{align*}
 \rcapa_{s,p}(E) & \leq \rcapa_{s,p}(E_0) + \lim_{j \to \infty} \sum_{i \geq j} \rcapa_{s,p}(A_i) \\
 & \leq C(n,p) \lim_{j \to \infty} \sum_{i \geq j} 2^{ip} \| u - u_i \|_{W^{s,p}(\R^n)}^p = 0.
\end{align*}
Secondly, if $x \in \R^n \setminus E$, then we have $x \notin E_0$ and there exists some $j_0 \in \N$ such that
$x \notin A_i$ for all $i \geq j_0$. Consequently, we have
\[
 M(u-u_i)(x) \leq 2^{-i} \quad \Rightarrow \quad |\widetilde{u}(x) - \widetilde{u}_i(x)| \leq 2^{-i} \qquad \text{for all } i \geq j_0 
\]
and thus the pointwise convergence $\lim_{i \to \infty} \widetilde{u}_i(x) = \widetilde{u}(x)$ holds. This finishes the proof of the lemma.
\end{proof}

As a direct consequence, by passing to a sufficiently fast convergent subsequence, we have the following

\begin{corollary}
\label{cor_strong_convergence_fractional_Sob}
Let $s \in (0,1)$, $p \in (1,\infty)$ and $u \in W^{s,p}(\R^n)$. Let $\{u_i\}_{i \in \N}$ be a sequence of functions in $W^{s,p}(\R^n)$ and suppose that $u_i \to u$ strongly in $W^{s,p}(\R^n)$. Then for a subsequence \textup{(}not relabeled\textup{)} we have pointwise convergence 
\[
\widetilde{u}_i(x) \to \widetilde{u}(x) \quad \text{for $\rcapa_{s,p}$-almost every } x \in \R^n.
\]
\end{corollary}

With Corollary~\ref{cor_strong_convergence_fractional_Sob} at hand, we can now address the announced pointwise convergence of bounded sequences in classical Sobolev spaces $W^{1,p}(\R^n)$, given for the purpose of comparison, or in the space $\BV(\R^n)$ of functions of bounded variation, as stated in Theorem~\ref{thm_bounded_BV_cap}.

\begin{theorem}
\label{thm_bounded_class_Sob_cap}
Let $p \in (1,n]$ and $u \in W^{1,p}(\R^n)$. Let $\{u_i\}_{i \in \N}$ be a
sequence in $W^{1,p}(\R^n)$ for which
$\{\Vert \nabla u_i\Vert_{L^p(\R^n;\R^n)}\}_{i\in\N}$ is bounded,
and suppose that $\widetilde{u}_i(x) \to \widetilde{u}(x)$ for $\mathcal L^n$-almost every $x\in \R^n$. Then there exists a set $E \subset \R^n$ with $\dim_{\mathcal H}(E) \leq n-p$ and such that for a subsequence \textup{(}not relabeled\textup{)} we have
\[
\widetilde{u}_i(x) \to \widetilde{u}(x) \quad \text{for every } x \in \R^n \setminus E.
\]
\end{theorem}

Before proving Theorem~\ref{thm_bounded_class_Sob_cap}, let us first notice that, similarly as Proposition~\ref{prop:pointwise to L1 convergence}, we can prove the following

\begin{proposition}\label{prop:pointwise to Lp convergence}
Let $p \in (1,n]$ and $u \in W^{1,p}(\R^n)$. Let $\{u_i\}_{i \in \N}$ be a
sequence in $W^{1,p}(\R^n)$ for which
$\{\Vert \nabla u_i\Vert_{L^p(\R^n;\R^n)}\}_{i\in\N}$ is bounded, and suppose that $u_i(x)\to u(x)$ 
for $\mathcal L^n$-almost every $x\in \R^n$. Then we have $u_i\to u$ in $L^{p^*}_{\loc}(\R^n)$ for every $q \in [1,p^*)$, where $p^*$ is the Sobolev conjugate of~$p$ defined as $\tfrac{np}{n-p}$ if $n>p$ and $\infty$ if $n=p$.
\end{proposition}

\begin{proof}[Proof of Theorem~\ref{thm_bounded_class_Sob_cap}]
Without loss of generality we can assume that $\{u_i\}_{i \in \N}$ is a sequence in $W^{1,p}(\R^n)$ where $u_i$ vanishes
\rep{}{in a neighborhood of $\R^n\setminus B(0,R)$ for some $R>0$ and}
for all $i \in \N$ (otherwise, we multiply by a suitable cut-off function $\eta_R$ with
\rep{}{$\mathbbm{1}_{B(0,R/2)} \leq \eta_R \leq \mathbbm{1}_{B(0,2R/3)}$}
for $R \in \N$, to obtain sequences which still have norm-bounded gradients and which coincide on $B(0,R/2)$ with the original sequence).

We first observe that by Proposition \ref{prop:pointwise to Lp convergence}, we actually
have strong convergence of the sequence $\{u_i\}_{i \in \N}$ in $L^p(\R^n)$. We then note that for every function $v \in W^{1,p}(\R^n)$ and every $s \in (0,1)$ the interpolation inequality 
\begin{equation*}
 [v]_{W^{s,p}(\R^n)}^p \leq \frac{2^{p(1-s)} n \omega_n}{s (1-s)p} \| v \|_{L^p(\R^n)}^{(1-s)p} \| \nabla v \|_{L^p(\R^n;\R^n)}^{sp}
\end{equation*}
is available, see \cite[Corollary 4.2]{BRASAL19}. Therefore, we even have strong convergence of the sequence $\{u_i\}_{i \in \N}$ in $W^{s,p}(\R^n)$, for every fixed $s \in (0,1)$. Thus, in view of Corollary~\ref{cor_strong_convergence_fractional_Sob}, we can select with a diagonal argument a subsequence (not relabeled) such that pointwise convergence $\widetilde{u}_i(x) \to \widetilde{u}(x)$ holds outside of a set $E_\ell \subset \R^n$ with $\rcapa_{1-1/\ell,p}(E_\ell)=0$, for each $\ell \in \N$. Defining the exceptional set as $E \coloneqq \bigcap_{\ell \in \N} E_\ell$, we then have pointwise convergence $\widetilde{u}_i(x) \to \widetilde{u}(x)$ for every $x \in \R^n \setminus E$. Furthermore, we deduce from Proposition~\ref{prop_cap_Hausdorff_dim} that $\mathcal H^d(E) \leq \mathcal H^d(E_\ell) = 0$, for every $d > n-p$ and $\ell > p/(d-n+p)$. This shows that the Hausdorff dimension of~$E$ is at most $n-p$, which completes the proof of the theorem.
\end{proof}

\begin{remark}
Let us note that the statements of Theorem~\ref{thm_Leb_frac_Sobolev}, Corollary~\ref{cor_strong_convergence_fractional_Sob} and Theorem~\ref{thm_bounded_class_Sob_cap} are trivially true in the case $s \in (0,1]$ and $p \in [1,\infty)$ with $sp >n$. By Morrey's embedding into H\"older spaces (see \cite[Theorem 8.2]{DINPALVAL12} for a version with fractional Sobolev spaces), the Lebesgue representative of a $W^{s,p}$-function is already globally H\"older continuous. Thus, every point of a $W^{s,p}$-function is a Lebesgue point, and for every bounded sequence in $W^{s,p}(\R^n)$ pointwise convergence of the Lebesgue representative of a subsequence is true on the full space, as a consequence of the Arzel\`a--Ascoli theorem.  
\end{remark}

\begin{proof}[Proof of Theorem~\ref{thm_bounded_BV_cap}]
Similarly as in the proof of Theorem~\ref{thm_bounded_class_Sob_cap}, we may here assume that $\{u_i\}_{i \in \N}$ is a sequence in $\BV(\R^n)$ where~$u_i$ vanishes in a neighborhood of $\R^n\setminus B(0,R)$ for some $R>0$ and for all $i \in \N$. 

By Proposition \ref{prop:pointwise to L1 convergence}, we get for the sequence
$\{u_i\}_{i \in \N}$ strong convergence in $L^{1}(\R^n)$,
and by the Sobolev embedding, also
boundedness in $L^{1^*}(\R^n)$.
We then note that for every function $v \in \BV(\R^n)$ and every $s \in (0,1)$ the interpolation 
\begin{equation*}
 [v]_{W^{s,1}(\R^n)} \leq \frac{2^{1-s} n \omega_n}{s (1-s)} \| v \|_{L^1(\R^n)}^{1-s} \big( |D v|(\R^n) \big)^{s},
\end{equation*}
is available, see \cite[Proposition 4.2]{BRALINPAR14}. Therefore, we get strong convergence of $\{u_i\}_{i \in \N}$ in $W^{s,1}(\R^n)$, for every fixed $s \in (0,1)$. We first consider the case $n>1$ and observe that, for any $p'$ chosen such that
\begin{equation*}
 1 < p' < \frac{n}{n-1} \frac{n + s -1}{n+s} <  \frac{n}{n-1} 
\end{equation*}
and $s'$ defined such that
\begin{align*}
 s'p' & = s - (n-1)(n+s)(p'-1) \\
 \Leftrightarrow \qquad p' & = \frac{n}{n-1} \frac{s-s'p'}{n+s} + 1 \frac{n+s'p'}{n+s},
\end{align*}
we have $0< s'p' < s$. We then note that the application of H\"older's inequality shows
\begin{align*}
 [v]_{W^{s',p'}(B(0,R))}^{p'} & = \int_{B(0,R)} \int_{B(0,R)} \frac{ |v(x)-v(\yn)|^{p'}}{|x-\yn|^{n+s'p'}}  \,d\mathcal L^n(x) \,d\mathcal L^n(\yn) \\
   & \leq \bigg( \int_{B(0,R)} \int_{B(0,R)} |v(x)-v(\yn)|^{\frac{n}{n-1}}  \,d\mathcal L^n(x) \,d\mathcal L^n(\yn) \bigg)^{\frac{s-s'p'}{n+s}} \\
   & \qquad \times \bigg( \int_{B(0,R)} \int_{B(0,R)} \frac{ |v(x)-v(\yn)|}{|x-\yn|^{n+s}}  \,d\mathcal L^n(x) \,d\mathcal L^n(\yn) \bigg)^{\frac{n+s'p'}{n+s}} \\
   & \leq \big[ \mathcal L^n(B(0,R)) \big]^{\frac{s-s'p'}{n+s}} \big[ 2 \|v\|_{L^{n/(n-1)}(B(0,R))}\big]^{\frac{n}{n-1} \frac{s-s'p'}{n+s}} [v]_{W^{s,1}(B(0,R))}^{\frac{n+s'p'}{n+s}} \\
   & = \big[ \mathcal L^n(B(0,R)) \big]^{(n-1)(p'-1)} \big[ 2 \|v\|_{L^{n/(n-1)}(B(0,R))} \big]^{n (p'-1)} [v]_{W^{s,1}(B(0,R))}^{n-(n-1)p'} 
\end{align*}
for every function $v \in W^{s,1}(B(0,R)) \cap L^{n/(n-1)}(B(0,R))$. Therefore, we also have strong convergence of the sequence $\{u_i\}_{i \in \N}$ in
$W^{s',p'}(B(0,R))$ and then in fact also in $W^{s',p'}(\R^n)$, see \cite[Lemma 5.1]{DINPALVAL12}, with $s',p'$ chosen as above and fixed $s \in (0,1)$.
By the arbitrariness of $s \in (0,1)$ and with $s'p' \nearrow s$ for $p' \searrow 1$, such choices are in particular possible for sequences $\{s_\ell\}_{\ell \in \N}$, $\{s'_\ell\}_{\ell \in \N}$ and $\{p'_\ell\}_{\ell \in \N}$ satisfying $s_\ell = 1-1/\ell$ and $s_\ell' p_\ell' = 1-2/\ell$
for each $\ell \in \N$. With Corollary~\ref{cor_strong_convergence_fractional_Sob}, we can therefore select with a diagonal argument a subsequence (not relabeled) such that pointwise convergence $\widetilde{u}_i(x) \to \widetilde{u}(x)$ holds outside of a set $E_\ell \subset \R^n$ with $\rcapa_{s_\ell',p_\ell'}(E_\ell)=0$, for each $\ell \in \N$. Setting $E \coloneqq \bigcap_{\ell \in \N} E_\ell$, we then have pointwise convergence $\widetilde{u}_i(x) \to \widetilde{u}(x)$ for every $x \in \R^n \setminus E$, while Proposition~\ref{prop_cap_Hausdorff_dim} implies that $\mathcal H^d(E) \leq \mathcal H^d(E_\ell) = 0$, for every $d > n-1$ and $\ell > 2/(d-n+1)$, which shows that the Hausdorff dimension of~$E$ is at most $n-1$ as claimed.

We finally comment on the case $n=1$. Here it turns out that for any $p' \in (1,\infty)$ and $s' \coloneqq s/p'$ we have 
\begin{equation*}
 [v]_{W^{s',p'}(B(0,R))}^{p'} \leq \big[ 2 \|v\|_{L^{\infty}(B(0,R))} \big]^{\frac{s}{s'}-1}  [v]_{W^{s,1}(B(0,R))}
\end{equation*}
for every function $v \in W^{s,1}(B(0,R)) \cap L^{\infty}(B(0,R))$. As the sequence $\{u_i\}_{i \in \N}$  is even bounded in $L^\infty(\R)$, this allows us to conclude that it converges strongly in $W^{s',p'}(\R)$. Proceeding as for $n>1$ we then find a set $E \subset \R$ with $\dim_{\mathcal H}(E) = 0$ such that the pointwise convergence $\widetilde{u}_i(x) \to \widetilde{u}(x)$ holds for every $x \in \R \setminus E$ for a suitable subsequence.
\end{proof}

\section{Pointwise convergence w.r.t.~diffuse measures}
\label{sec_pointwise_convergence_slicing}

In this section we first prove, via the slicing technique, that each bounded sequence 
in~$\BV(\R^n)$ admits a subsequence converging outside of a $\nu$-vanishing set, where~$\nu$  
is a generalized product measure of diffuse type. We then obtain Theorem~\ref{thm:main BV theorem} as a direct consequence. For convenience of notation and reference to the literature we here prefer to always work with the precise representatives, even though the main arguments exclude jump points of the sequence and the limit function so that the Lebesgue representatives would be suitable as well.

\begin{theorem}\label{thm:convergence theorem}
Let $u\in\BV(\R^n)$. Let $\{u_i\}_{i \in \N}$ be a sequence 
in $\BV(\R^n)$ for which $\{|D u_i|(\R^n)\}_{i \in \N}$ is
bounded, and suppose that $u_i^*(x)\to u^*(x)$ for $\mathcal L^n$-almost
every $x\in \R^n$.
Let $\nu$ be a positive measure of finite mass admitting a representation
\[
\nu=\mathcal L^{n-1}\otimes \nu_\ys
\]
with $\nu_\ys(\{t\})=0$ for all $t\in\R$, for $\mathcal L^{n-1}$-almost
every $\ys\in\R^{n-1}$.
Then for a subsequence \textup{(}not relabeled\textup{)} we
have $u_i^*(x)\to u^*(x)$ for $\nu$-almost every $x\in \R^n$.
\end{theorem}

\begin{proof}
We divide the proof into two steps.
\subsubsection*{Step 1.}
First we consider the one-dimensional case.
Let~$\nu_0$ be a positive measure of finite mass on~$\R$
such that $\nu_0(\{x\})=0$ for every $x\in\R$. 
Take $u_i,u\in \BV(\R)$, $i\in\N$, with $u_i^*(x)\to u^*(x)$ for 
$\mathcal L^1$-almost every $x\in \R$. (Here we do not assume that 
$\{|D u_i|(\R)\}_{i \in \N}$ is
bounded). We consider the set 
\begin{align*}
 M^1 \coloneqq  \Big \{ x \in \R \colon & x \notin \Big(J_u\cup\bigcup_{i\in\N}J_{u_i}\Big)
 \text{ with } u_i^*(x)\to u^*(x) \text{ as } i \to \infty, \\
 & \nu_0((x,\infty)), \nu_0((-\infty,x))\le (2/3)\nu_0(\R) \Big \}
\end{align*}
and observe that $M^1$ is non-empty, as the convergence
$u_i^*(x)\to u^*(x)$  takes place $\mathcal L^1$-almost everywhere,
the approximate jump sets $J_u,J_{u_i}$ are at most countable, and the conditions
$\nu_0((x,\infty)),\nu_0((-\infty,x))\le (2/3)\nu_0(\R)$
are satisfied on a non-empty interval in $\R$ (as $\nu_0$
does not charge singletons). Note that we could also work with a
convenient countable and dense subset $G \subset \R$ as admissible
points for this splitting procedure, in the sense that we could choose
it such that for all $x \in G$ there hold
$x \notin (J_u\cup\bigcup_{i\in\N}J_{u_i})$ and $u_i^*(x)\to u^*(x)$
as $i \to \infty$, and then consider instead of $M^1$ the non-empty
set of points $x \in G$ such that
$\nu_0((x,\infty)), \nu_0((-\infty,x))\le (2/3)\nu_0(\R)$
are satisfied (as done later in Step 2a).

We now pick an arbitrary point $x_1^1 \in M^1$. Next we split each of the intervals $(-\infty,x_1^1)$ and $(x_1^1, \infty)$ into two parts as above, by considering accordingly defined sets $M^2_1 \subset (-\infty,x_1^1)$, $M^2_3 \subset (x_1^1,\infty)$, picking arbitrary points $x^2_1 \in M^2_1$, $x^2_3 \in M^2_3$, and keeping $x^2_2 \coloneqq  x^1_1$.
Continuing like this, we get a monotonously increasing sequence of collections of points
\[
\{x_j^{\ell}\}_{j=1}^{2^{\ell}-1}\subset \R\setminus \Big(J_u\cup \bigcup_{i\in\N}J_{u_i}\Big),\quad \ell\in\N,
\]
with
(denote $x^{\ell}_0=-\infty$ and $x^{\ell}_{2^{\ell}}=\infty$)
\begin{equation}\label{eq:decay of diffuse measure}
\max_{j\in\{0,\ldots,2^{\ell}-1\}} \nu_0((x^{\ell}_j, x^{\ell}_{j+1}))
\le (2/3)^{\ell}\nu_0(\R)
\end{equation}
and $u_i^*(x^{\ell}_j)\to u^*(x^{\ell}_j)$ as $i\to\infty$, for each $\ell\in\N$ and $j=1,\ldots,2^{\ell}-1$.

Denote
\[
\alpha^\ell_i \coloneqq \max_{j\in\{0,\ldots,2^{\ell}-1\}} |u^*_i(x^{\ell}_j)-u^*(x^{\ell}_j)|
\]
so that for any fixed $\ell\in \N$, we have $\alpha^\ell_i \to 0$ as $i\to\infty$.
Note that since necessarily $u^*(x)\to 0$ and $u_i^*(x)\to 0$ as $x\to -\infty$,
we interpret $|u^*_i(x^{\ell}_0)-u^*(x^{\ell}_0)|=0$, where $x^{\ell}_0=-\infty$.
We have for every $\ell\in\N$
\begin{equation}\label{eq:convergence in case n=1}
\begin{split}
&\int_{\R}|u_i^*-u^*|\,d \nu_0
= \sum_{j=0}^{2^{\ell}-1}\int_{x_j^{\ell}}^{x_{j+1}^{\ell}}|u_i^*-u^*|\,d \nu_0\\
&\ \le \sum_{j=0}^{2^{\ell}-1}\nu_0((x_j^{\ell},x_{j+1}^{\ell}))\left(|u_i^*(x^{\ell}_j)- u^*(x^{\ell}_j)|+|D (u_i-u)|((x^{\ell}_j, x^{\ell}_{j+1}))\right)\quad\textrm{by }
\eqref{eq:fundamental theorem of calculus for BV}\\
&\ \le\sum_{j=0}^{2^{\ell}-1}\nu_0((x_j^{\ell},x_{j+1}^{\ell}))\left( \alpha^\ell_i +|D (u_i-u)|((x^{\ell}_j, x^{\ell}_{j+1}))\right)\\
&\ \le \sum_{j=0}^{2^{\ell}-1}\nu_0((x^{\ell}_j, x^{\ell}_{j+1})) \alpha^\ell_i  +(2/3)^{\ell}
\nu_0(\R)\sum_{j=0}^{2^{\ell}-1}|D (u_i-u)|((x^{\ell}_j, x^{\ell}_{j+1}))\quad\textrm{by }
\eqref{eq:decay of diffuse measure}\\
&\ \le \nu_0(\R) \alpha^\ell_i +(2/3)^{\ell} 
\nu_0(\R)\left(|D u_i|(\R)+|D u|(\R)\right).
\end{split}
\end{equation}

We will use this estimate to prove the general case.

\subsubsection*{Step 2.}

Now we consider the general case. Let $\{u_i\}_{i \in \N},u$ be as given 
in the statement of the theorem.

For $\mathcal L^{n-1}$-almost every $\ys\in \R^{n-1}$ we have that
$u_\ys^n, (u_i)_\ys^n\in \BV(\R)$ for $i \in \N$ --- recall the notation 
from Section~\ref{subsec:one dimensional sections}. For simplicity, we will 
discard the superscript~$n$ and write simply $u_\ys, (u_i)_\ys$.

We have $u_i^*(x)\to u^*(x)$ for $\mathcal L^n$-almost every $x\in\R^n$, implying 
that for $\mathcal L^{n-1}$-almost every $\ys\in \R^{n-1}$, $u_i^*(\ys,t)\to u^*(\ys,t)$ 
for almost every $t\in\R$. In view of~\eqref{eq:sections and jump sets} this implies 
that for $\mathcal L^{n-1}$-almost every $\ys\in \R^{n-1}$, $((u_i)_\ys)^*(t)\to (u_\ys)^*(t)$ 
for almost every $t\in\R$. Thus, for $\mathcal L^{n-1}$-almost every $\ys\in\R^{n-1}$, the functions 
$(u_i)_\ys$, $u_\ys$ satisfy the assumptions for the application of Step 1.

\subsubsection*{Step 2a.}

We next reason that the collection of points $\{t_j^{\ell}(\ys)\}$ selected in Step~1 can be chosen 
to be $\mathcal L^{n-1}$-measurable with respect to $\ys \in \R^{n-1}$. 
For this purpose we observe that the set
\[
 \Big\{ x \in \R^n \colon x \notin \Big(J_u\cup\bigcup_{i\in\N}J_{u_i}\Big)
 \text{ with } u_i^*(x)\to u^*(x) \text{ as } i \to \infty \Big\}
\]
is Borel and its complement is of vanishing $\mathcal L^{n}$-measure.
Hence, by Fubini, we can select a countable and dense subset 
\[ 
 G = \{g_1,g_2, g_3, \ldots\} \subset \R
\]
of points and an $\mathcal L^{n-1}$-negligible set $B_0 \subset \R^{n-1}$ such that
\[
 \begin{cases}
  (u_i)_\ys, u_\ys \in \BV(\R), \text{ for all } i \in \N, \\
  ((u_i)^*)_\ys(t)=((u_i)_\ys)^*(t), (u^*)_\ys(t)=(u_\ys)^*(t), \text{ for all } i \in \N, \\
  ((u_i)_\ys)^*(t)\to (u_\ys)^*(t) \text{ as } i \to \infty, \\
  (\ys,t) \notin J_u\cup\bigcup_{i\in\N}J_{u_i}\quad\textrm{and so}\quad
  t\notin J_{u_\ys}\cup\bigcup_{i \in \N} J_{(u_i)_\ys}
 \end{cases}
\]
for all $\ys \in \R^{n-1} \setminus B_0$ and $t \in G$.
We next consider the sets
\begin{align*}
 B_j \coloneqq  \Big\{ \ys \in \R^{n-1} \colon & \ys \notin B_i \text{ for all } i \in \{0,1,\ldots,j-1\}, \\
  & \nu_\ys((g_j,\infty)), \nu_\ys((-\infty,g_j))\le (2/3)\nu_\ys(\R) \Big \},
\end{align*}
which are $\mathcal L^{n-1}$-measurable in $\R^{n-1}$ by weak*
$\mathcal L^{n-1}$-measurability of the mapping $\ys \mapsto \nu_\ys$
and disjoint by construction. As~$G$ was chosen dense in~$\R$,
we hence have the decomposition 
\[
 \R^{n-1} = \bigcup_{j\in\N_0} B_j
\]
(as already commented on in Step 1). If we now define a function $t^1_1 \coloneqq  \sum_{j \in \N} g_j \mathbbm{1}_{B_j}$, 
then we immediately observe that, as a step function, it is $\mathcal L^{n-1}$-measurable, and for each 
$\ys \in \R^{n-1} \setminus B_0$ the point $t^1_1(\ys)$ belongs to the set $M^1_\ys$ (defined
just as the set $M^1$ in Step 1). Iterating this splitting procedure, we then arrive at a 
monotonously increasing sequence of collections of points
\[
\{t_j^{\ell}(\ys)\}_{j=1}^{2^{\ell}-1}\subset \R\setminus \Big(J_{u_\ys}\cup \bigcup_{i\in\N}J_{(u_i)_\ys}\Big),\quad \ell\in\N,
\]
which are $\mathcal L^{n-1}$-measurable with respect to the variable~$\ys$. As a consequence, due to the Borel 
measurability of the precise representatives, also the function
\[
\alpha^\ell_i(\ys) \coloneqq \max_{j\in\{0,\ldots,2^{\ell}-1\}} |((u_i)_\ys)^*(t^{\ell}_j(\ys))-(u_\ys)^*(t^{\ell}_j(\ys))|
\]
is $\mathcal L^{n-1}$-measurable with respect to~$\ys$. Let us also note that because of the convergence 
$((u_i)_\ys)^*(t) \to (u_\ys)^*(t)$  as $i \to \infty$, for all $\ys \in \R^{n-1} \setminus B_0$ and $t \in G$, we again 
have $\alpha^\ell_i(\ys)\to 0$ as $i\to\infty$, for fixed $\ell\in\N$ and all $\ys \in \R^{n-1} \setminus B_0$.

\subsubsection*{Step 2b.}

We next apply the estimate~\eqref{eq:convergence in case n=1} from Step 1. This shows that for 
$\mathcal L^{n-1}$-almost every $\ys\in \R^{n-1}$ (those in $\R^{n-1} \setminus B_0$) we have
\begin{equation}\label{eq:estimate for slices}
\begin{split}
&\int_{\R}|((u_i)_\ys)^*(t)-(u_\ys)^*(t)|\,d \nu_\ys(t)\\
&\qquad \qquad \le \nu_\ys(\R) \alpha^\ell_i(\ys)
+(2/3)^{\ell}\nu_\ys(\R)\left(|D (u_i)_\ys|(\R)+|D u_\ys|(\R)\right).
\end{split}
\end{equation}
Fix $\eps>0$. We initially note that by definition of $\nu$,
\[
\nu(\R^n)=\int_{\R^{n-1}}\nu_\ys(\R)\,d\mathcal L^{n-1}(\ys).
\]
Therefore, on the one hand, by choosing a constant $M_\eps > 0$ sufficiently large, we can assume that
\[
\int_{A_0^\eps} \nu_\ys(\R)\,d\mathcal L^{n-1}(\ys) < \eps \quad \text{for} \quad A_0^\eps \coloneqq  \{\ys \in \R^{n-1} \colon \nu_\ys(\R) > M_\eps \},
\]
and on the other hand the weighted measure $\nu_\ys(\R) \, d\mathcal L^{n-1}$ is a finite measure on~$\R^{n-1}$. By Egorov's theorem, applied for
each fixed $\ell\in \N$ to the measure $\nu_\ys(\R) \,d\mathcal L^{n-1}$
and the sequence of measurable functions
$\{\ys \mapsto \alpha^\ell_i(\ys)\}_{i \in \N}$ converging pointwisely to
zero, we can find a measurable set $A_\ell^\eps\subset \R^{n-1}$ with 
\[
 \int_{A_\ell^\eps} \nu_\ys(\R)\,d\mathcal L^{n-1}(\ys) <2^{-\ell}\eps
\]
and a sequence of positive numbers $\{\bar{\alpha}^\ell_i\}_{i \in \N}$ with
$\bar{\alpha}^\ell_i \to 0$ as $i \to \infty$ and such that 
$\alpha^\ell_i(\ys)\le \bar{\alpha}^\ell_i$ 
for all $i\in \N$, for all $\ys\in \R^{n-1}\setminus A_\ell^\eps$. 
Let $A^\eps\coloneqq \bigcup_{\ell \in \N_0} A_\ell^\eps$, with 
\[
 \int_{A^\eps} \nu_\ys(\R)\,d\mathcal L^{n-1}(\ys) < 2\eps.
\]
Employing~\eqref{eq:sections and jump sets},~\eqref{eq:estimate for slices} and finally~\eqref{eq:slice representation for total variation} for $k=n$, we then find
\begin{equation}\label{eq:long calculation}
\begin{split}
& \int_{(\R^{n-1}\setminus A^\eps)\times \R}|u_i^{*}-u^{*}|\,d (\mathcal L^{n-1}\otimes \nu_\ys)\\
&\qquad= \int_{\R^{n-1}\setminus A^\eps}\int_{\R} |u_i^{*}(\ys,t)-u^{*}(\ys,t)|\,d\nu_\ys(t)\,d\mathcal L^{n-1}(\ys)\\
&\qquad=  \int_{\R^{n-1}\setminus A^\eps}\int_{\R} |((u_i)_\ys)^{*}(t)-(u_\ys)^{*}(t)|\,d\nu_\ys(t)\,d\mathcal L^{n-1}(\ys)\\
&\qquad\le 
\int_{\R^{n-1}\setminus A^\eps} \Big[\nu_\ys(\R) \alpha^\ell_i(\ys)
+(2/3)^{\ell}\nu_\ys(\R)\big(|D (u_i)_\ys|(\R)
+|D u_\ys|(\R)\big)\Big]\,d\mathcal L^{n-1}(\ys) \\
&\qquad\le \bar{\alpha}^\ell_i\int_{\R^{n-1}}\nu_\ys(\R)\,d\mathcal L^{n-1}(\ys)\\
&\qquad\qquad+(2/3)^{\ell}\sup_{\ys\in \R^{n-1}\setminus A^\eps}
\nu_\ys(\R)\int_{\R^{n-1}}[|D (u_i)_\ys|(\R)
+|Du_\ys|(\R)]\,d\mathcal L^{n-1}(\ys)\\
&\qquad \leq \bar{\alpha}^\ell_i \nu(\R^n)+(2/3)^{\ell} M_\eps (\left|D_n u_i|(\R^n)+|D_n u|(\R^n)\right).
\end{split}
\end{equation}
Thus, we obtain
\begin{align*}
&\limsup_{i\to\infty}\int_{(\R^{n-1}\setminus A^\eps)\times \R}|u_i^{*}-u^{*}|\,d (\mathcal L^{n-1}\otimes \nu_\ys)\\
&\qquad\qquad\le
(2/3)^{\ell} M_\eps \limsup_{i\to\infty}(\left|D_n u_i|(\R^n)+|D_n u|(\R^n)\right)\\
&\qquad\qquad\le
(2/3)^{\ell} M_\eps \limsup_{i\to\infty}(\left|D u_i|(\R^n)+|D u|(\R^n)\right) \qquad\to 0 \quad \textrm{as } \ell \to\infty,
\end{align*}
since $\{|D u_i|(\R^n)\}_{i \in \N}$ is a bounded sequence. By passing to a subsequence (not relabeled), 
we have $u_i^{*}(x)\to u^{*}(x)$ for $\mathcal L^{n-1}\otimes \nu_\ys$-almost every
$x\in (\R^{n-1}\setminus A^\eps)\times \R$.
We can do this for sets $A^\eps=A^{1/j}$ with 
\[
 \int_{A^{1/j}} \nu_\ys(\R)\,d\mathcal L^{n-1}(\ys) < \frac{2}{j},
\]
for $j\in\N$, and by a diagonal argument we obtain that for any $j\in\N$, $u_i^{*}(x)\to u^{*}(x)$ for
$\mathcal L^{n-1}\otimes \nu_\ys$-almost every $x\in (\R^{n-1}\setminus A^{1/j})\times \R$, that is,
$u_i^{*}(x)\to u^{*}(x)$ for $\mathcal L^{n-1}\otimes \nu_\ys$-almost every $x\in \R^{n}$, i.e.,
$\nu$-almost every $x\in \R^{n}$.
\end{proof}

\begin{proof}[Proof of Theorem \ref{thm:main BV theorem}]
By~\eqref{eq:slice representation for total variation},
we have
\[
|D_n^d w|=\mathcal L^{n-1}\otimes |D^d w^n_\ys|,
\]
where $|D^d w^n_\ys|(\{t\})=0$ for every $t\in\R$, for $\mathcal L^{n-1}$-almost every $\ys\in\R^{n-1}$.
Thus, by Theorem \ref{thm:convergence theorem} we find a subsequence of $\{u_i\}_{i \in \N}$ 
(not relabeled) such that $u_i^{*}(x)\to u^{*}(x)$ for $|D_n^d w|$-almost every $x\in \R^{n}$.
Passing to further subsequences (not relabeled), we then we obtain (after a change of coordinates) for every $k=1,\ldots,n$ that
$u_i^{*}(x)\to u^{*}(x)$ for $|D_k^d w|$-almost every $x\in \R^{n}$. Noting that
\[
|D^d w|\le \sum_{k=1}^n |D^d_k w|,
\]
we hence have shown the pointwise convergence $u_i^{*}(x)\to u^{*}(x)$
for $|D^d w|$-almost every $x\in \R^{n}$.
\end{proof}

\section{Remarks and examples}
\label{sec_remarks_examples}

In this section we give some remarks concerning Theorems~\ref{thm_bounded_BV_cap} 
and~\ref{thm:main BV theorem} and examine their sharpness.

\begin{remark}
\label{remark_assumptions_weak_star_convergence}
If $\{u_i\}_{i \in \N}$ is a sequence in $\BV(\R^n)$ with $u_i\to u$ weakly* 
in $\BV(\R^n)$, then it is also norm-bounded in $\BV(\R^n)$ 
(see e.g. \cite[Proposition 3.13]{AFP}), and of course for a subsequence we 

have $u_i^*(x)\to u^*(x)$ for $\mathcal L^n$-almost every $x\in\R^n$.
Thus, the assumptions of Theorems~\ref{thm:Hausdorff main BV theorem}, \ref{thm:main BV theorem},
and~\ref{thm:convergence theorem} are satisfied.
\end{remark}

In Theorems~\ref{thm_bounded_BV_cap} and~\ref{thm:main BV theorem}, 
pointwise convergence $u_i^{*}\to u^{*}$ is stated outside of an exceptional 
set~$E \subset \R^n$ that satisfies $\dim_{\mathcal H}(E) \leq n-1$ and $|D^d w|(E)=0$,
respectively. Apart from some specific situations, this cannot be improved to 
$\mathcal H^{n-1}(E) =0$ or $|D w|(E)=0$, meaning that we in particular 
need to exclude the jump part $D^j w$.

\begin{example}
Let $w = u \coloneqq \mathbbm{1}_{[0,1]}\in\BV(\R)$ and define
\[
u_i(x)\coloneqq \max\{0, \min\{1,1/4+ix\}\}\mathbbm{1}_{(-\infty,1]}(x),\quad i\in\N.
\]
Then it is easy to see that $\{u_i\}_{i \in \N}$ is a norm-bounded sequence in $\BV(\R)$ 
with $u_i\to u$ weakly* and even strictly in $\BV(\R)$. However, we have $u_i^*(0)\equiv 1/4\not\to 1/2=u^*(0)$. Moreover, $u^+(0)=1$ and $u^-(0)=0$, so $u_i^*(0)$ does not converge
to these either. Here $|D^j u|(\{0\})=\mathcal H^{0}(\{0\})=1$.

On the other hand, for any dimension  $n\in\N$ and in the special case that the
functions $u_i$ are defined as convolutions of a function $u \in \BV(\R^n)$ 
(with standard mollifiers), we have $u_i\to u$ strictly in $\BV(\R^n)$ and $u_i^*(x)\to u^*(x)$ for $\mathcal H^{n-1}$-almost every $x\in\R^n$ and thus $|Du|$-almost every $x\in\R^n$, see
\cite[Theorem 3.9 \& Corollary 3.80]{AFP}.
\end{example}

In Theorems~\ref{thm_bounded_BV_cap} and~\ref{thm:main BV theorem} the pointwise convergence 
occurs outside of an $(n-1)$-dimensional set and a $|D^d w|$-negligible set, respectively,
but it is not clear how large exactly such exceptional sets can be. In this regard, consider the following

\begin{example}
Let $\{q_j\}_{j \in \N}$ be an enumeration of the rational points on the real line, and define
$E^i_{j}\coloneqq (q_j-1/i, q_j+1/i)$ and 
\[
u_i(x)\coloneqq \sum_{j\in \N}2^{-j}\mathbbm{1}_{E^i_{j}}(x),\quad x\in\R,\ i\in\N.
\]
Then clearly $u_i\searrow u:\equiv 0$ $\mathcal L^1$-almost everywhere, and
\[
|Du_i|(\R)\le \sum_{j\in \N}2^{-j}|D \mathbbm{1}_{E^i_{j}}|(\R)
=\sum_{j\in \N}2^{-j+1}= 2.
\]
Thus, $\{u_i\}_{i \in \N}$ is a norm-bounded sequence in $\BV(\R)$ and the assumptions 
of Theorems~\ref{thm_bounded_BV_cap} and~\ref{thm:main BV theorem} are satisfied. 
However, $u^*_i(q_j)\ge 2^{-j}\not\to 0 =u^*(q_j)$ as $i \to \infty$, for every $j\in\N$. 
Thus, the pointwise convergence $u_i\to u$ (for any subsequence) fails in a fairly 
large set, though this set is still $\sigma$-finite with respect to the Hausdorff 
measure $\mathcal H^{n-1}=\mathcal H^{0}$.
It is not clear whether the exceptional can in some cases be larger than this.
In Section~\ref{sec:decreasing} we will show that it is always at most 
$\sigma$-finite with respect to $\mathcal H^{n-1}$ in two special cases: 
when $n=1$ and when $\{u_i\}_{i \in \N}$ is a decreasing sequence.
\end{example}

One can also ask whether it is necessary to pass to a subsequence
in Theorems~\ref{thm_bounded_BV_cap} and~\ref{thm:main BV theorem}; 
the answer is in general yes.

\begin{example}
Let $C=\bigcap _{k \in \N} C_k\subset [0,1]$ be the standard $1/3$-Cantor set,
where $C_0 = [0,1]$ and~$C_{k}$ is obtained iteratively from~$C_{k-1}$ by removing 
the open middle third of each interval, meaning that in the end~$C_k$ consists 
of $2^k$ compact intervals $C^k_{1},\ldots,C^k_{2^k}$, each of length $3^{-k}$. 
Note that $\dim_{\mathcal H}(C) = \log_3(2)$. 
Let $\{E_i\}_{i \in \N}$ be the sequence of sets
\[
C^1_{1},C^1_{2},C^2_{1},C^2_{2},C^2_{3},C^2_{4},C^3_{1},\ldots.
\]
Let $v$ be the Cantor--Vitali function, let
$w=u\coloneqq v \mathbbm 1_{[0,1]}\in \BV(\R)$, and define
\[
u_i\coloneqq u+\mathbbm{1}_{E_i},\quad i\in\N.
\]
Then $u_i\to u$ in $L^1(\R)$ and $\mathcal L^1$-almost everywhere,
and $|Du_i|(\R)=4$ for all $i\in\N$, hence,  $\{u_i\}_{i \in \N}$ is a norm-bounded 
sequence in $\BV(\R)$. However, for all $x\in C$, and thus for all $x$ in the support 
of $|D^c u|$, we have that $x\in E_i$ for infinitely many $i\in\N$. For such $i \in \N$ 
we have $u_i^*(x)\ge u^*(x)+1/2$ and thus $u_i^*(x)$ fails to converge
to $u^*(x)$. Hence, it is necessary to pass to a subsequence in order to obtain 
pointwise convergence outside of $(n-1)$-dimensional or $|D^d u|$-negligible sets.  
\end{example}

Theorems~\ref{thm_bounded_BV_cap} and~\ref{thm:main BV theorem} both involve an
exceptional set~$E$ where the precise representatives 
(for a subsequence) do not converge, but they are different in nature. 
Theorem~\ref{thm_bounded_BV_cap} gives the upper bound $n-1$ on the Hausdorff dimension
of~$E$ and hence neglects $(n-1)$-dimensional sets, while Theorem~\ref{thm:main BV theorem} 
states that~$E$ is in particular $|D^c u|$-vanishing, and here~$D^c u$ can be
supported on an $(n-1)$-dimensional set.

\begin{example}
Let $C=\bigcap _{k \in \N} C_k\subset [0,1]$ be a generalized Cantor set, where $C_0 = [0,1]$ 
and~$C_{k}$ is obtained from~$C_{k-1}$ by removing an open set in the middle of each interval 
of fraction $1-2 \cdot 3^{-2k+1}$, meaning that in the end~$C_k$ consists of $2^k$ compact 
intervals of length $3^{-k^2}$. This generalized Cantor set is uncountable by construction,
exactly as the standard $1/3$-Cantor set. Concerning its Hausdorff dimension, let us note that 
for each fixed $d >0$, the Hausdorff pre-measure of fineness $\delta >0$ satisfies
\begin{align*} 
\mathcal H_\delta^d (C) & = \inf \Big\{ \sum_{j \in \N} (\diam U_j)^d \colon 
C \subset \bigcup_{j \in \N} U_j, \diam U_j < \delta \text{ for all } j \in \N \Big\} \\
  & \leq 2^k [3^{-k^2}]^d  \quad \stackrel{k \to \infty}{\longrightarrow} 0 ,
\end{align*}
by taking the intervals from~$C_k$ as an admissible covering of $C$, for $k^2 > - \log_3 \delta$. This shows, 
for each fixed $d >0$, that $\mathcal H^d (C) = 0$, implying $\dimens_{\mathcal H}(C) = 0$.
Introducing for each $k \in \N$ the piecewise affine and monotone functions
\[
 u_k(x) = 2^{-k} 3^{k^2} \int_0^x \mathbbm{1}_{C_k}(t) \,dt \quad \text{for } x \in [0,1], 
\]
we easily verify that
\[
 \max_{x \in [0,1]} | u_{k+1}(x) - u_k(x)| = 2^{-k-1} ( 1-2 \cdot 3^{-2k-1}) < 2^{-k-1}.
\]
Thus, $\{u_i\}_{i \in \N}$ is a Cauchy sequence in $C([0,1])$ and converges uniformly to a 
continuous, monotone and bounded function $u \in C([0,1])$, the generalized Cantor--Vitali 
function. We observe $u \in \BV(0,1)$, and since $u_k$ is constant on each connected component 
of $[0,1] \setminus C$ (for $k$ sufficiently large), we finally conclude that $Du$ is 
concentrated on the $0$-dimensional set $C$ and that it is purely Cantor, i.e. $Du = D^c u$. 
\end{example}

Note that we have the following consequence of the proof of Theorem~\ref{thm:convergence theorem} 
in the one-dimensional setting.

\begin{proposition}
Let $u\in\BV(\R)$.  Let $\{u_i\}_{i \in \N}$ be a sequence 
in $\BV(\R)$ for which $\{|D u_i|(\R)\}_{i \in \N}$ is
bounded, and suppose that $u_i^*(x)\to u^*(x)$ for $\mathcal L^1$-almost
every $x\in \R^n$. Let $w\in\BV(\R)$. Then we have
\[
\int_{\R}|u_i^*-u^*|\,d|D^d w|\to 0 \quad \text{as } i\to\infty.
\]
\end{proposition}

\begin{proof}
Equation~\eqref{eq:convergence in case n=1} with $\nu_0=|D^d w|$ gives
\begin{align*}
&\limsup_{i\to\infty}\int_{\R}|u_i^*-u^*|\,d |D^d w|\\
&\qquad \le \limsup_{i\to\infty}\big(|D^d w|(\R)\alpha^{\ell}_i+(2/3)^{\ell} |D^d w|(\R)(|D u_i|(\R)+|D u|(\R))\big)\\
&\qquad = (2/3)^{\ell} |D^d w|(\R)\limsup_{i\to\infty}\big(|D u_i|(\R)+|D u|(\R)\big),
\end{align*}
which becomes arbitrarily small as $\ell\to\infty$.
\end{proof}

\begin{example}
In general dimensions $n \ge 2$ we cannot have
\[
\int_{\R^n}|u_i^*-u^*|\,d|D^d u|\to 0  \quad \text{as } i\to\infty,
\]
even if $u_i\to u$ strongly in $\BV(\R^n)$, since~$u^*$ need not be integrable with respect 
to~$|D^d u|$. This can be seen by considering $u(x)\coloneqq \eta(x)|x|^{-1/2}$ in~$\R^2$, 
where~$\eta$ is a smooth function with $\mathbbm{1}_{B(0,1)} \leq \eta \leq \mathbbm{1}_{B(0,2)}$.
Then for $u_i\coloneqq \min\{u,i\}$, $i\in\N$, we have $u_i\to u$ in $\BV(\R^2)$ but
\[
\int_{\R^2}|u_i^*-u^*|\,d|D^d u|
= \pi \int_0^{i^{-2}}(r^{-1/2}-i)r^{-1/2}\,dr=\infty
\]
for all $i\in\N$.
\end{example}

\begin{example}\label{ex:liftings}
The motivation for this paper arose from the theory of \emph{liftings},
see~\cite{RiSh}. Let $\Om\subset \R^n$ be an open, bounded set. A lifting is a measure
$\gamma\in \mathcal M(\Om\times \R^N;\R^{N\times n})$ for which there exists
a function $u\in \BV(\Om;\R^N)$ with integral average $0$ and such that
the chain rule formula
\begin{align*}
&\int_{\Om}\nabla_x \varphi(x,u(x))\,d\mathcal L^n(x)\\
&\qquad +\int_{\Om\times \R^N}\nabla_\yn\varphi(x,\yn)\,d\mathcal \gamma(x,\yn)
=0\quad\textrm{for all }\varphi\in C_0^1(\Om\times \R^N)
\end{align*}
holds. Here $u=:[\gamma]$ can be shown to be unique.
Given a function $u\in \BV(\Om;\R^N)$ with integral average $0$,
an \emph{elementary lifting} $\gamma[u]$ is defined by
\[
\langle \varphi,\gamma[u] \rangle
\coloneqq \int_{\Om}\int_0^1 \varphi(x,u^\theta(x))\,d\theta\,dDu(x)
\quad\textrm{for all }\varphi\in C_0^1(\Om\times \R^N),
\]
where $u^{\theta}(x)\coloneqq \widetilde{u}(x)$ for $x\in \R^n\setminus S_u$ and
\[
u^{\theta}(x)\coloneqq \theta u^{-}(x)+(1-\theta)u^+(x)\quad\textrm{for }x\in J_u.
\]
The family of \emph{approximable liftings} is then defined as the weak*
limits in $\mathcal M(\Om\times \R^N;\R^{N\times n})$ of sequences of elementary 
liftings. Thus, for an approximable lifting~$\gamma$ we have a sequence 
$\{u_i\}_{i \in \N}$ in $\BV(\Om;\R^N)$ with integral averages $0$ and  
$\gamma[u_i]\overset{*}{\rightharpoonup}\gamma$ in 
$\mathcal M(\Om\times \R^N;\R^{N\times n})$, and then we also have
$u_i\overset{*}{\rightharpoonup} u\coloneqq [\gamma]$ in $\BV(\Om;\R^N)$.
Therefore, it is of interest to better understand weak* convergence in the
BV space, and we expect that the results of the current paper may
be of use in further research on (approximable) liftings.
In particular, Theorem \ref{thm:main BV theorem} may be of help
in investigating the decomposition of approximable liftings into mutually 
singular measures, which are related to the measures $D^a u$, $D^c u$, $D^s u$;
see \cite[Theorem 3.11]{RiSh} for an existing structure theorem.
\end{example}

\section{Two special cases}\label{sec:decreasing}

In this section we show that we can obtain pointwise convergence outside
of an exceptional set with $\sigma$-finite $\mathcal H^{n-1}$-measure 
in two special cases: when $n=1$, and when $\{u_i\}_{i \in \N}$ is a 
decreasing sequence. We start with the first case.

\begin{proposition}\label{prop:1d proof}
Let $u\in \BV(\R)$. Let $\{u_i\}_{i \in \N}$ be a sequence in $\BV(\R)$ 
for which $\{|D u_i|(\R)\}_{i \in \N}$ is
bounded, and
suppose that $u_i^*(x)\to u^*(x)$ for $\mathcal L^1$-almost every $x\in \R$. 
Then there exists an at most countable set $E \subset \R$ such that for a 
subsequence \textup{(}not relabeled\textup{)} we have
$u_i^*(x)\to u^*(x)$ for every $x\in \R\setminus E$.
\end{proposition}

\begin{remark}
Note that this proposition may appear to be an improvement over Step 1
of the proof of Theorem \ref{thm:convergence theorem}. However,
unlike in Step 1, in this proposition we already need to pass to a subsequence.
Therefore, it is not clear how one would apply this proposition in Step 2, since
the subsequence could be different for each one-dimensional section $(u_i)_\ys$.
\end{remark}

\begin{proof}
Since the sequence of measures $\{Du_i\}_{i \in \N}$ is mass-bounded,
passing to a subsequence (not relabeled) we find a positive finite measure~$\nu$ on~$\R$ 
such that $|Du_i|\overset{*}{\rightharpoonup}\nu$. 
Note that $\nu \ge|Du|$ (see \cite[Proposition 1.62]{AFP}).
Let~$E$ be the set of singletons charged by~$\nu$; then~$E$ is at most countable,
so that we can write
\[
  E=\{a_1,a_2,\ldots\}.
\]
Let $M\in\N$ and define $\alpha_M\coloneqq \sum_{k>M}\nu(a_k)$. Fix an arbitrary compact 
set $K\subset \R\setminus \{a_1,\ldots,a_M\}$ and $\eps>0$. Denote $\mu\coloneqq \nu|_{\R\setminus E}$, 
so that $\mu$ does not charge singletons. By a similar splitting procedure as in Step 1 of the proof 
of Theorem~\ref{thm:convergence theorem}, we find closed intervals 
$I_j\subset \R\setminus \{a_1,\ldots,a_M\}$, $j=1,\ldots,L \in\N$, such that
\[
	K\subset \bigcup_{j=1}^{L} I_j
\]
and
\[
	\mu(I_j)<\eps\quad\textrm{for each }j=1,\ldots,L,
\]
and, denoting by $x_j$ the left end point of the interval~$I_j$, such that 
$x_j\notin J_u\cup\bigcup_{i\in\N}J_{u_i}$ and $|u_i^*(x_j)-u^*(x_j)|\to 0$ as 
$i\to\infty$, for each $j=1,\ldots,L$. 	Then
\[
	\nu(I_j)\le\alpha_M +  \mu(I_j) \le \alpha_M+\eps
\]
for each $j=1,\ldots,L$. By the weak* convergence $|Du_i|\overset{*}{\rightharpoonup}\nu$, we have
\[
	\limsup_{i\to\infty}|Du_i|(I_j)\le \nu(I_j)
\]
for each $j=1,\ldots,L$. Thus, for every $x\in I_j$, 
recalling~\eqref{eq:fundamental theorem of calculus for BV 2} and $|Du| \leq \nu$, we get
\begin{align*}
	\limsup_{i\to\infty}|u_i^*(x)-u^*(x)|
	&\le \limsup_{i\to\infty} \big(|u_i^*(x_j)-u^*(x_j)|+|D(u_i-u)|(I_j)\big)\\
	&\le \limsup_{i\to\infty}|Du_i|(I_j)+|Du|(I_j)\\
	&\le \nu(I_j)+\nu(I_j)\\
	&\le 2\alpha_M+2\eps.
\end{align*}
Letting $\eps\to 0$ and exhausting the set $\R\setminus \{a_1,\ldots,a_M\}$
with compact subsets $K$, we get
\[
	\limsup_{i\to\infty}|u_i^*(x)-u^*(x)|\le 2\alpha_M
\]
for all $x\in \R\setminus \{a_1,\ldots,a_M\}$, Thus, letting $M\to \infty$, we finally
end up with
\[
	\limsup_{i\to\infty}|u_i^*(x)-u^*(x)|=0 
\]
for all $x\in \R\setminus E$, which completes the proof of the proposition.
\end{proof}

We next examine the second case, where we deal with a decreasing sequence.

\begin{proposition}
Let $u\in \BV(\R^n)$. Let $\{u_i\}_{i \in \N}$ be a decreasing sequence in $\BV(\R^n)$
for which $\{|D u_i|(\R^n)\}_{i \in \N}$ is
bounded, and suppose that
$u_i(x)\to u(x)$ for $\mathcal L^n$-almost every $x\in \R^n$. Then  there exists a set $E \subset \R^n$ such that~$E$ is $\sigma$-finite with respect to~$\mathcal H^{n-1}$ and
$u_i^*(x)\to u^*(x)$ for every $x\in \R^n\setminus E$.
\end{proposition}

\begin{proof}
By Proposition \ref{prop:pointwise to L1 convergence}, we have $u_i\to u$ in $L_{\loc}^1(\R^n)$, and we can in fact assume 
$u_i\to u$ in $L^1(\R^n)$  (otherwise, we multiply the sequence by cut-off functions $\eta_R$ with 
$\mathbbm{1}_{B(0,R/2)} \leq \eta_R \leq \mathbbm{1}_{B(0,R)}$ for $R \in \N$).

First we assume that $u\equiv 0$. Let
\[
	P\coloneqq \bigcup_{i \in \N}\big(S_{u_i}\setminus J_{u_i}\big),
\]
so that $\mathcal H^{n-1}(P)=0$, and let $E\subset \R^n$ be the set where the convergence
$u_i^*\to u^*$ fails.
Then $E=\bigcup_{k=1}^{\infty}E_k$ with
\[
	E_k\coloneqq \{x\in \R^n \colon \lim_{i\to\infty}u^*_i(x)\ge 2/k\}.
\]
Fix $k\in\N$. By the coarea formula~\eqref{eq_coarea}, we have for every $i\in\N$
\begin{equation}\label{eq:using coarea}
	\int_{1/k}^{2/k}|D\mathbbm{1}_{\{u_i>t\}}|(\R^n)\,dt
	\le \int_{-\infty}^{\infty}|D\mathbbm{1}_{\{u_i>t\}}|(\R^n)\,dt=|Du_i|(\R^n).
\end{equation}
Thus, for every $i\in\N$ we can choose 	$t_i\in (1/k,2/k)$ such that
\[
	|D\mathbbm{1}_{\{u_i^*>t_i\}}|(\R^n)\le k|Du_i|(\R^n).
\]
Now, setting $S_i\coloneqq \{x\in\R^n \colon u_i^*(x)> t_i\}$, we have 
$\mathbbm{1}_{S_i}\to 0$ in $L^1(\R^n)$ and
\begin{equation}\label{eq:Ei perimeter finite limsup}
	\limsup_{i\to\infty}|D\mathbbm{1}_{S_i}|(\R^n)
	\le k\limsup_{i\to\infty}|Du_i|(\R^n)<\infty.
\end{equation}
Fix $i\in\N$ and let $x\in E_k\setminus P$. Then, by the fact that $\{u_i\}_{i \in \N}$ is 
a decreasing sequence, we have $x\in S_i$. Using the fact that~$x$ is either a Lebesgue or 
a jump point for~$u_i$ by definition of the set~$P$, we can further verify 
\[
	\lim_{r\to 0}\frac{\mathcal L^n(B(x,r)\cap S_i)}{\mathcal L^n(B(x,r))}
	\geq \frac{1}{2}.
\]
Setting	$R_i\coloneqq \omega_n^{-1/n}(3\mathcal L^n(S_i))^{1/n}$,

for all $r\ge R_i$ we have 
\[
	\frac{\mathcal L^n(B(x,r)\cap S_i)}{\mathcal L^n(B(x,r))}\le \frac{1}{3}.
\]
Thus, by continuity we find $0<r_x\le R_i$ such that
\[
	\frac{\mathcal L^n(B(x,r_x)\cap S_i)}{\mathcal L^n(B(x,r_x))}=\frac{1}{3}.
\]
By the relative isoperimetric inequality (with constant~$C_P$ depending only on~$n$), 
cp. \cite[Remark 3.45]{AFP}, we have
\begin{equation}\label{eq:rel isop ineq}
	\frac{\mathcal L^n(B(x,r_x))}{r_x}\le 3C_P |D \mathbbm{1}_{S_i}|(B(x,r_x)).
\end{equation}
The collection $\{B(x,r_x)\}_{x\in E_k\setminus P}$ is a covering of $E_k\setminus P$.
By the Vitali $5$-covering theorem, we then find a countable collection of disjoint balls
$\{B(x_j,r_j)\}_{j \in \N}$ such that the balls
$\{B(x_j,5r_j)\}_{j \in \N}$ cover $E_k\setminus P$. Thus,
using~\eqref{eq:rel isop ineq}, we find for the $(n-1)$-dimensional Hausdorff pre-measure 
of fineness $10R_i$ of the set $E_k \setminus P$  the estimate 
\begin{align*}
	\mathcal H^{n-1}_{10R_i}(E_k\setminus P)\le
	\omega_{n-1} \sum_{j \in \N} (5r_j)^{n-1}
	&\le 3 C_P 5^{n-1} \frac{\omega_{n-1}}{\omega_n}\sum_{j \in \N} |D\mathbbm{1}_{S_i}|(B(x_j,r_j))\\
	&\le C(n) |D\mathbbm{1}_{S_i}|(\R^n).
\end{align*}
Letting $i\to\infty$, so that also $R_i\to 0$, by~\eqref{eq:Ei perimeter finite limsup} we get
\[
	\mathcal H^{n-1}(E_k\setminus P)
	\le C(n)\limsup_{i\to\infty}|D\mathbbm{1}_{S_i}|(\R^n)<\infty.
\]
Since this holds for each $k \in \N$, we obtain that the set 
$E\setminus P=\bigcup_{k \in \N} E_k\setminus P$
is $\sigma$-finite with respect to~$\mathcal H^{n-1}$, and then so is~$E$.
Note that for every $x\in \R^n\setminus E_k$, we have $\lim_{i\to\infty}u^*_i(x)< 2/k$.
Thus, for every $x\in \R^n\setminus E$, we have $\lim_{i\to\infty}u^*_i(x)=0=u^*(x)$, 
completing the proof in the case $u\equiv 0$.
	
In the general case, $\{u_i-u\}_{i \in \N}$ is a decreasing sequence in $\BV(\R^n)$ 
for which $\{|D(u_i-u)|(\R^n)\}_{i \in \N}$ is bounded, and $u_i-u\to 0$ holds 
$\mathcal L^n$-almost everywhere. By the first part, we have $(u_i-u)^*\to 0$ outside 
of a set that is $\sigma$-finite with respect to~$\mathcal H^{n-1}$.
Note that outside of the $\mathcal H^{n-1}$-negligible set
\[
	(S_{u}\setminus J_{u})\cup
	\bigcup_{i \in \N} \big(S_{u_i}\setminus J_{u_i}\big)
\]
we have $(u_i-u)^*=u_i^*-u^*$. Therefore, the assertion of the proposition follows. 
\end{proof}

\paragraph{{\bf Acknowledgments.}}The authors wish to thank Giles Shaw for
posing the question that led to this research and for discussions on
the topic, and also Jan Kristensen and Bernd Schmidt for discussions.

\end{document}